\documentclass[11pt]{amsart}
\usepackage{amsfonts,amsmath,amssymb,latexsym,soul,cite,mathrsfs,enumitem}

\usepackage{color,enumitem,graphicx}
\usepackage[colorlinks=true,urlcolor=blue,
citecolor=red,linkcolor=blue,linktocpage,pdfpagelabels,
bookmarksnumbered,bookmarksopen]{hyperref}

\usepackage{amsthm}
\usepackage{commath}
\usepackage[english]{babel}

\usepackage[hyperpageref]{backref}

\usepackage[left=2.9cm,right=2.9cm,top=2.8cm,bottom=2.8cm]{geometry}
\usepackage{graphicx}
\usepackage{dsfont}
\newtheorem{lemma}{Lemma}  
\newtheorem{proposition}{Proposition}
\newtheorem{theorem}{Theorem}

\newtheorem{remark}{Remark}

\numberwithin{equation}{section}

\title[Normalized solutions to a SBP system]{Normalized solutions to a  Schr\"odinger-Bopp-Podolsky system
under  Neumann boundary conditions}

\author[D. G. Afonso]{Danilo G. Afonso}
\author[G. Siciliano]{Gaetano Siciliano}

\address[D. G. Afonso]{\newline\indent Departamento de Matem\'atica
\newline\indent
 Universidade de Bras\'ilia 
\newline\indent
Campus Universit\'ario Darcy Ribeiro,  70910-900 Bras\'ilia, DF, Brazil }
\email{\href{mailto:d.g.afonso@mat.unb.br}{d.g.afonso@mat.unb.br}}

\address[G. Siciliano]{\newline\indent Departamento de Matem\'atica
\newline\indent 
Instituto de Matem\'atica e Estat\'istica
\newline\indent 
 Universidade de S\~ao Paulo 
\newline\indent 
Rua do Mat\~ao 1010,  05508-090 S\~ao Paulo, SP, Brazil }
\email{\href{mailto:sicilian@ime.usp.br}{sicilian@ime.usp.br}}

\thanks{D. G. Afonso was supported by CNPq grant 132634/2018-0. G. Siciliano
was supported by Fapesp grant 2018/17264-4, CNPq grant 304660/2018-3 and Capes.}
\subjclass[2010]{Primary 35J50, 35J58; Secondary 35Q55, 35Q61} 
\keywords{Schr\"odinger-Bopp-Podolsky system, Krasnoselskii genus, Lagrange multipliers, weak solutions.}

\makeindex

\begin{document}

\maketitle

\begin{abstract}
In this paper we study a Schr\"odinger-Bopp-Podolsky system of partial differential equations 
in a bounded and smooth
 domain of $\mathbb R^3$  with a non constant coupling factor. Under a compatibility condition 
 on the boundary data 
we deduce existence and multiplicity of solutions by means of the  Ljusternik-Schnirelmann theory.
\end{abstract}


\section{Introduction}

The Schr\"{o}dinger-Newton equation consists of a nonlinear coupling of the Schr\"{o}dinger equation with a gravitational potential of newtonian form, representing the interaction of a particle with its own gravitational field. 

In 1998, Benci and Fortunato \cite{BenciFortunato1998} treated a similar problem, 
where the Schr\"odinger equation was coupled with Maxwell's equations. Such coupling represents the interaction of 
the particle with its own electromagnetic field. The coupling factor is a constant
$q\neq0$.
In their paper the authors consider  standing waves 
solutions in the purely electrostatic field and this  leads to the so-called 
Schr\"odinger-Poisson system.
They impose a  Dirichlet boundary condition both on the matter field $u$ and the 
electrostatic field $\phi$ and employed variational methods and 
critical point theory to develop a 
procedure that would become standard to treat other similar problems. 

Later, Pisani and Siciliano \cite{PisaniSiciliano2013} treated a Schr\"{o}dinger-Poisson 
system with Neumann boundary conditions on the scalar field $\phi$ and considered 
the case in which the interaction factor responsible for the coupling of the equations 
is non-constant. This gives rise to important and interesting 
considerations regarding the geometry of the  manifold on which find the solutions.

\medskip

In this paper we treat a modification of the problem dealt with by Pisani and Siciliano 
consisting in the addition of a biharmonic term in the equation of the electrostatic potential and imposing
appropriate boundary conditions. The problem studied  can be interpreted as a 
coupling of the
Schr\"{o}dinger equation with the Bopp-Podolsky electrodynamics (for more details 
on this  matter, see \cite{dAveniaSiciliano2019} and the 
references therein). However here we focus on the mathematical aspects of the 
problem.

We point out that in the  literature there are few papers
concerning Schr\"odinger-Bopp-Podolski  systems. Beside \cite{dAveniaSiciliano2019}
we cite here \cite{CT,LPT} where the authors study the problem with a critical nonlinearity,
\cite{GG} where solutions with {\sl a priori} given interaction 
energy for the Schr\"odinger-Bopp-Podolski  system  are found,
\cite{H} where the problem has been addressed in the context of
closed $ 3-$dimensional manifolds both in the subcritical and  critical case
and 
\cite{GK} where the fibering method has been used to prove existence
results depending on a parameter and also nonexistence.

Coming back to our problem, the aim here is to study the following system of 
partial differential equations in a connected, bounded, smooth open set $\Omega 
\subset \mathbb R^3$:
\begin{align}
- \Delta u + q \phi u  - \kappa |u|^{p-2}u  & = \omega u \quad \text{ in } \ \Omega \label{eq:1}\\
\Delta^2 \phi - \Delta \phi & = q u^2 \quad \text{ in } \ \Omega \label{eq:2}
\end{align}
in the unknowns $u,\phi:\Omega\to \mathbb R$ and $\omega\in \mathbb R$. Here
 $\kappa \in \mathbb R$ and $q:\Omega\to \mathbb R$ are given. We assume the following boundary conditions:
\begin{align}
u & = 0 \quad \text{ on } \ \partial \Omega \label{eq:b1}\\
\dpd{\phi}{\mathbf n} & = h_1 \quad \text{ on } \ \partial \Omega \label{eq:b2} \\
\dpd{\Delta \phi}{\mathbf n} & = h_2 \quad \text{ on } \ \partial \Omega \label{eq:b3}
\end{align}
and for simplicity we assume $h_1, h_2$ continuous. 
The symbol $\mathbf n$ 
denotes the unit 
vector normal to $\partial \Omega$ pointing outwards. Since $u$ represents physically the amplitude of the wave function of a particle confined in $\Omega$, we 
assume the following normalizing condition:
\begin{equation}
\label{eq:normalization_condition}
\int_{\Omega} u^2 dx = 1.
\end{equation}
We also assume that the coupling factor $q$ is continuous on $\overline \Omega$:
\begin{equation}\label{eq:continuity}
q \in C(\overline{\Omega}).
\end{equation}

Our main result is the following:

\begin{theorem}
\label{thm:main_theorem}
Let $\kappa > 0$, $p \in (2, 10/3)$ and 
\begin{equation}\label{eq:alfa}
\alpha := \int_{\partial \Omega} h_2 ds - \int_{\partial \Omega} h_1 ds.
\end{equation} 
Assume that $\inf_\Omega q < \alpha < \sup_\Omega q$ and that $|q^{-1}(\alpha)| = 0$. 
Then there exist
infinitely many solutions $(u_n, \omega_n, \phi_n) \in H_0^1 (\Omega) \times \mathbb R \times H^2(\Omega)$ to the problem \eqref{eq:1}  and  \eqref{eq:2} under conditions \eqref{eq:b1}-\eqref{eq:continuity}, with 
$$
\int_{\Omega} |\nabla u_n|^2 dx \to + \infty.
$$
Moreover the ground state solution can be assumed positive.
\end{theorem}

Our approach is variational, indeed the solutions will be found as critical points of an energy functional
restricted to  a suitable constraint. In this context by a ground state solution we mean the solution with minimal energy. 
Moreover as a byproduct of the proof, we obtain that also the energy of these solutions is divergent.

\begin{remark}
It is easy to see that if $\kappa < 0$ the result holds with $p \in (2, 6)$. For $\kappa = 0$, see \cite{Afonso2020a}.
\end{remark}

The paper is organized as follows.

In Section \ref{sec:aux} we introduce an auxiliary problem
which will be useful in order to deal with homogeneous boundary conditions.

In Section \ref{sec:M} we give some properties of the constraint $M$
on which we will find the solution.

In Section \ref{sec:variational} we introduce the energy functional 
and show that its critical points on $M$ will give solutions of the problem.

In the final Section \ref{sec:final} by implementing the Ljusternick-Schnirelmann theory
we prove Theorem \ref{thm:main_theorem}.

\medskip

As a matter of notations,
we use the letters $c, c',\ldots $ to denote positive constant whose value can change from line to line.
We use $\|\cdot\|_{p}$ to denote the standard $L^{p}-$norm.

\section{An auxiliary problem}\label{sec:aux}
Our aim is to define a functional whose critical points will be the weak solutions to the problem. 
In order to deal with homogeneous boundary conditions, that will permit to write the functional
in a simpler form, we make a change of variable.

Consider the following auxiliary problem (where $\alpha$ is defined in \eqref{eq:alfa})
\begin{align}
\Delta^2 \chi - \Delta \chi &= \alpha/|\Omega| \quad \text{ in } \ \Omega,\label{eq:chi1} \\
\dpd{\chi}{\mathbf n} &= h_1 \quad \text{ on } \ \partial \Omega,  \label{eq:chi2}\\
\dpd{\Delta\chi}{\mathbf n} &= h_2 \quad \text{ on } \ \partial \Omega, \label{eq:chi3}\\
\int_{\Omega} \chi dx &= 0. \label{eq:chi4}
\end{align}
It is easy to see it has a unique solution.
Indeed, let $\theta$ be the unique function satisfying 
\begin{align*}
\Delta \theta - \theta &= \alpha/|\Omega| \quad \text{ in } \ \Omega \medskip\\
\dpd{\theta}{\mathbf n} &= h_2 \quad \text{ on } \ \partial \Omega,\medskip\\
\int_{\Omega} \theta dx &= \int_{\partial \Omega} h_1 ds
\end{align*}
and then let $\chi$ the  unique function which satisfies
\begin{align*}
\Delta \chi &= \theta, \quad \text{ in } \ \Omega \medskip\\
\dpd{\chi}{\mathbf n} &= h_1 \quad \text{ in } \ \partial \Omega \\
\int_{\Omega} \chi dx &= 0, 
\end{align*} 
see e.g. \cite{Taylor}.
Then  it is easy to see that by construction $\chi$ satisfies  \eqref{eq:chi1}-\eqref{eq:chi4}.


The change of variables we make is 
$$
\varphi = \phi - \chi - \mu, 
$$
where 
$$
\mu = \frac{1}{|\Omega|}\int_{\Omega} \phi dx.
$$
With the new variables $(u, \omega, \varphi, \mu)$ our problem becomes
\begin{align}
- \Delta u + q(\chi + \varphi)u - \kappa|u|^{p - 2}u & = \omega u - \mu q u\quad \text{ in } \ \Omega, \label{eq:new_problem_1} \medskip\\ 
\Delta^2 \varphi - \Delta \varphi & = qu^2 - \alpha/|\Omega| \quad \text{ in } \ 
\Omega, \label{eq:new_problem_2}  \medskip\\
u & = 0 \quad \text{ on } \ \partial \Omega, \label{eq:new_problem_3} \medskip\\
\int_{\Omega} u^2 dx & = 1, \label{eq:new_problem_4}  \medskip\\
\dpd{\varphi}{\mathbf n} & = 0 \quad \text{ on } \ 
\partial \Omega, \label{eq:new_problem_5} \medskip \\
\dpd{\Delta \varphi}{\mathbf n} & = 0 \quad \text{ on } \ \partial \Omega, \label{eq:new_problem_6} \medskip \\
\int_{\Omega} \varphi dx &= 0. \label{eq:new_problem_7}
\end{align}
Notice that the compatibility condition between \eqref{eq:new_problem_2}, \eqref{eq:new_problem_5} and \eqref{eq:new_problem_6} now reads as 
$$
\int_{\Omega} qu^2 dx = \alpha.
$$
Let us define the sets
\begin{align*}
S & := \cbr{u \in H_0^1(\Omega) \ : \ \int_{\Omega} u^2 dx = 1}, \\
N & := \cbr{u \in H_0^1(\Omega) \ : \ \int_{\Omega} qu^2 dx = \alpha}, \\
M & := S \cap N.
\end{align*}
Recall that $\alpha$ depends on both the boundary conditions to the original problem.

If the problem has a solution, then of course $M \neq \emptyset$. Hence, 
\begin{equation}
\label{eq:range_for_alpha}
q_{\min} \leq \alpha \leq q_{\max}
\end{equation}
where
$$
q_{\min} = \inf_\Omega q \quad \text{ and } \quad q_{\max} = \sup_\Omega q.
$$
Indeed, if $\alpha < q_{\min}$, then
$$
\alpha = \int_{\Omega} qu^2 dx \geq  q_{\min} > \alpha,  
$$
which is a contradiction. The case $\alpha > q_{\max}$ is analogous.

From \eqref{eq:range_for_alpha} we deduce that $q^{-1}(\alpha)$ is not empty, 
and indeed is its measure that will play a major role. 

Suppose $\alpha = q_{\min}$ and $|q^{-1}(\alpha)| = 0$. Then 
\begin{align*}
\int_{\Omega} qu^2 dx &= \int_{\{ x\in \Omega : q(x)>\alpha\}} q u^2 dx >\alpha,
\end{align*}
so $M=\emptyset$.
If $\alpha = q_{\max}$ and $|q^{-1}(\alpha)| = 0$ we proceed in an analogous manner to conclude that $M$ is empty and so the problem has no solutions. Therefore, we arrive at the following necessary condition for the existence of solutions: either 
\begin{equation}
\label{eq:necessary_condition_1}
q_{\min} < \alpha < q_{\max}
\end{equation}
or
\begin{equation}
\label{eq:necessary_condition_2}
|q^{-1}(\alpha)| \neq 0.
\end{equation}

\section{The manifold $M$}\label{sec:M}

We now state some properties of the set $M$, referring the reader to \cite{PisaniSiciliano2013} 
for the omitted proofs.

We first note that $M$ is symmetric with respect to the origin: if $u \in M$ then $-u \in M$. This follows trivially from the definition of $M$. We also note that $M$ is weakly closed in $H_0^1 (\Omega)$.

Now, we want to show that under condition \eqref{eq:necessary_condition_1} the set $M$ is not empty. For this, we introduce the following notation.

Let $A \subset \Omega$ be an open subset and define
$$
S_A := \cbr{u \in H_0^1 (A) \ : \ \int_A u^2 dx = 1}
$$
and 
$$
g_A: u \in S_A \mapsto \int_A qu^2 dx \in \mathbb R.
$$

It is immediately seen that
\begin{equation*}
\label{eq:inclusion_1}
g_A(S_A) \subset [\inf_A q, \sup_A q].
\end{equation*}

\begin{lemma}
The following inclusion holds:
\begin{equation*}
\label{eq:inclusion_2}
(\inf_A q, \sup_A q) \subset \overline{g_A(S_A)}.
\end{equation*}
\end{lemma}

We can conclude the following:

\begin{proposition}
\label{prop:M_is_not_empty}
Let $A \subset \Omega$ be an open subset. If $\alpha \in (\inf_A q, \sup_A q)$ then there exists $u \in H_0^1(A)$ such that 
$$
\int_A u^2 dx = 1 \quad \text{ and } \quad \int_A qu^2 dx = \alpha. 
$$
\end{proposition}

In particular by taking $A=\Omega$ we get
\begin{theorem}
Assume that $\inf_\Omega q < \alpha < \sup_\Omega q$. Then $M$ is not empty.
\end{theorem}
%
%

Let us recall the definition of genus of Krasnoselki.
Given $A$ a closed and symmetric subset of some Banach space, with $0\notin A$, the {\sl genus }
of $A$ is denoted as $\gamma(A)$ and defined as the least integer  $k$
for which there exists  a continuous and even map $h:A\to \mathbb R^{k}\setminus\{0\}$. 
By definition it is $\gamma(\emptyset) = 0$ and if it is not possible to construct continuous odd maps from $A$
to any $\mathbb R^{k}\setminus\{0\}$, it is set $\gamma(A) = +\infty$.

 It is  known that the genus is a topological invariant (by odd homeomorphism)
and that the genus of the  sphere in $\mathbb R^{N}$ is $N$.

The next result says that  $M$ has subsets of arbitrarily large genus.

\begin{theorem}
\label{thm:existence_of_sets_of_genus_k}
Let $u_1, \ldots, u_k \in M$ be functions with disjoint supports and let 
$$
V_k = \langle u_1, \ldots, u_k \rangle
$$
be the space spanned by $u_1, \ldots, u_k$. Then $M\cap V_k$ is the $(k-1)$-dimensional sphere, hence $\gamma(M \cap V_k) = k$.
\end{theorem}

Now, it is natural if one raises the question of whether there exists such functions with disjoint supports for arbitrary $k$. The answer is positive:

\begin{theorem}
If \eqref{eq:necessary_condition_1} holds then for every $k \geq 2$ there exist $k$ functions $u_1, \ldots, u_k \in M$ with disjoint supports. Hence $\gamma(M) = + \infty$.
\end{theorem}

Let
\begin{align*}
& G_1 : u \in H_0^1 (\Omega) \mapsto \int_{\Omega} u^2 dx - 1 \in \mathbb R, \\
& G_2 : u \in H_0^1(\Omega) \mapsto \int_{\Omega} qu^2 dx - \alpha \in \mathbb R
\end{align*}
and 
$$
G = (G_1, G_2).
$$
Then
$$
M = \cbr{u \in H_0^1(\Omega) \ : \ G_1(u) = G_2(u) = 0} = G^{-1}(0).
$$

We note that $G$ is of class $C^1$.

Let us show, for the reader convenience, that $G_1'(u)$ and $G_2'(u)$ are linearly independent, so $G$ will be a submersion and $M$ a submanifold of 
codimension $2$.

\begin{proposition}
Assume $M$ is not empty. The differentials $G_1'(u)$ and $G_2'(u)$ are linearly independent for every $u \in M$ if and only if 
\begin{equation}
\label{eq:measure_is_zero}
|q^{-1}(\alpha)| = 0.
\end{equation}
\end{proposition}

\begin{proof}
First, assume \eqref{eq:measure_is_zero}. We will show that $G_1'(u)$ and $G_2'(u)$ are linearly independent, for all $u \in M$. Suppose that there are $a, b \in \mathbb R$ such that 
$$
a G_1'(u) + bG_2'(u) = 0 \quad \text{ in } H^{-1} (\Omega)
$$ 
for some $u \in M$. Evaluating this expression in $u$, we find that $a + b \alpha = 0$. Then
$$
aG_1'(u)[v] + b G_2'(u)[v] = b \del{-\alpha \int_{\Omega} uv dx + \int_{\Omega} quv dx } = 0 \quad \forall v \in H_0^1(\Omega),
$$
that is, 
$$
b \int_{\Omega} (q - \alpha)uv dx = 0 \quad \forall v \in H_0^1 (\Omega).
$$
If $b \neq 0$ then we would have $(q - \alpha)u = 0$ a.e., and hence, in view of \eqref{eq:measure_is_zero}, $u = 0$, a contradiction. Thus $G_1'(u)$ and $G_2'(u)$ are linearly independent for all $u \in M$.

Now, suppose \eqref{eq:measure_is_zero} is not satisfied. Then $q^{-1}(\alpha)$ has not empty interior, hence there is some test function $u$ with support in $q^{-1}(\alpha)$ such that $||u||_2 = 1$. It is immediately seen that $u \in M$ because $qu = \alpha u$ and hence 
$$
G_2'(u) = \alpha G_1'(u),
$$
which completes the proof.
\end{proof}

\section{Variational setting}\label{sec:variational}

We now proceed to study the variational framework of the problem. Our aim is to construct a functional whose critical points will be the weak solutions of the problem.

Following \cite{GazzolaGrunauSweers2010}, let 
$$
V = \cbr{\xi \in H^2(\Omega) \ : \ \dpd{\xi}{\mathbf n} = 0 \text{ on } \partial \Omega}.
$$
We remark that $V$ is a closed subspace of $H^2(\Omega)$. Indeed, let $\{v_n\} \subset V$ such that $v_n \to v$ in $V$. Then $0 = \gamma_1(v_n) \to \gamma_1(v)$ and hence $\gamma_1(v) = 0$, where $\gamma_1$ denotes the trace operator which for smooth functions gives the directional derivative in the direction of the exterior normal on the boundary. Being a closed subspace, $V$ inherits the Hilbert space structure of $H^2(\Omega)$.

Recall that 
$$
\varphi = \phi - \chi - \mu
$$
where 
$$
\mu = \frac{1}{|\Omega|} \int_{\Omega} \phi dx.
$$
In this way, we have $\overline \varphi = 0$,
where from now on, given a function $f$, we denote with $\overline f$ its average in $\Omega$. 
 Consider then the following natural decomposition of $V$:
\begin{equation}
\label{eq:decomposition_of_V}
V = \widetilde V \oplus \mathbb R
\end{equation}
where 
$$
\widetilde V = \cbr{\eta \in V \ : \ \overline \eta = 0}.
$$
On $\widetilde V$ we have the equivalent norm
$$
||\eta||_{\widetilde V} = \del{||\nabla \eta||_2^2 + ||\Delta \eta||_2^2}^{1/2}.
$$

Consider the functional $F:H_0^1(\Omega) \times H^2 (\Omega)$ defined below:
\begin{align*}
F(u, \varphi) = & \frac12 \int_{\Omega} |\nabla u|^2 dx + \frac12 \int_{\Omega} q(\varphi + \chi)u^2 dx - \frac{\kappa}{p} \int_{\Omega} |u|^p dx\\
& - \frac{1}{4} \int_{\Omega} (\Delta \varphi)^2 dx - \frac{1}{4} \int_{\Omega} |\nabla \varphi|^2 dx - \frac{\alpha}{2 |\Omega|} \int_{\Omega} \varphi dx.
\end{align*}
It is easy to see that this functional is of class $C^1$ and that given $u \in H_0^1(\Omega)$ and $\varphi \in H^2 (\Omega)$ we have
\begin{align*}
F'_u(u, \varphi)[v] & = \int_{\Omega} \nabla u  \nabla v dx + \int_{\Omega} q(\varphi + \chi)uv dx  - \kappa \int_{\Omega} |u|^{p - 2}u v dx \\
F'_\varphi(u, \varphi)[\xi] &= \frac12 \int_{\Omega} q \xi u^2 dx - \frac12 \int_{\Omega} \Delta \varphi \Delta \xi dx - \frac12 \int_{\Omega} \nabla \varphi  \nabla \xi dx - \frac{\alpha}{2|\Omega|}\int_{\Omega} \xi dx   
\end{align*}
for every $v \in H_0^1(\Omega)$ and $\xi \in H^2(\Omega)$.

Then,
$(u, \varphi, \omega, \mu)\in H^1_0(\Omega)\times H^2(\Omega)\times \mathbb R\times \mathbb R$ is a weak solution to \eqref{eq:new_problem_1}-\eqref{eq:new_problem_7} if
and only if 
\begin{equation}\label{eq:defsol}
\begin{array}{ll}
& (u, \varphi) \in M \times \widetilde V, \\
& \forall v \in H_0^1 (\Omega):  \ F_u'(u, \varphi)[v] =\displaystyle \omega \int_{\Omega} uv dx - \mu \int_{\Omega} q u v dx, \\
& \forall \xi \in V: \ F_\varphi'(u, \varphi)[\xi] = 0.
\end{array}
\end{equation}

\begin{theorem}
Let $(u, \varphi) \in H_0^1(\Omega) \times H^2 (\Omega)$. Then there exist $\omega, \mu \in \mathbb R$ such that $(u, \varphi, \omega, \mu)$ is a solution to \eqref{eq:new_problem_1}-\eqref{eq:new_problem_7} if and only if $(u, \varphi)$ is a critical point of $F$ constrained on $M \times \widetilde V$, in which case the real constants $\omega, \mu$ are the two Lagrange multipliers with respect to $F'_u$.
\end{theorem}

\begin{proof}
Indeed  $(u, \varphi)$ is a critical point of $F$ constrained on $M \times \widetilde V$ if 
and only if
\begin{align*}
& \forall v \in T_uM:  \ F_u'(u, \varphi)[v] = 0, \\
& \forall \xi \in \widetilde V: \ F_\varphi'(u, \varphi)[\xi] = 0.
\end{align*}
Note that the tangent space to $\widetilde V$ at $\varphi$ is $\widetilde V$ itself.


Then a weak solution, according to \eqref{eq:defsol} and the Lagrange multipliers rule,
 is a constrained critical point. 

Suppose on the contrary that $(u, \varphi)$ is a constrained critical point. Then, 
again by the  Lagrange multipliers rule, we have that there exists $\omega, \mu\in \mathbb R$
such that 
$$
\forall v\in H^1_0(\Omega): F_u'(u, \varphi)[v] = \omega \int_{\Omega} uv dx - \mu \int_{\Omega} quv dx.
$$
It remains to prove that $F_\varphi'(u, \varphi)[\xi] = 0$ for all $\xi \in V$. But this follows by the 
 decomposition \eqref{eq:decomposition_of_V}, noticing that $F_\varphi'(u, \varphi)[r] = 0$ for every constant $r \in \mathbb R$.
 Then \eqref{eq:defsol} is satisfied and this concludes the proof.
\end{proof}

The functional $F$ constrained on $M \times \widetilde V$ is unbounded from above and from below. 
 This issue has been addressed by Benci and Fortunato \cite{BenciFortunato1998} and in many subsequent papers. Their standard reduction argument goes as follows:


\begin{enumerate}
\item[(i)] For every fixed $u\in H^{1}_{0}(\Omega)$ there exists a unique $\Phi(u)$ such that $F'_\varphi(u,  \Phi(u)) = 0$.

\item[(ii)] The map $u \mapsto \Phi(u)$ is of class $C^1$.

\item[(iii)] The graph of $\Phi$ is a manifold, and we are reduced to study the functional $J(u) = F(u, \Phi(u)$, possibly constrained.
\end{enumerate}

However the method sketched above fails in our situation, for two reasons. First, we see that $F_\varphi'(u, \varphi) = 0$ with $\varphi \in \widetilde V$ is just 
\begin{align*}
\Delta^2 \varphi - \Delta \varphi - qu^2 + \alpha/|\Omega| &= 0 \quad \text{ in } \ \Omega, \\
\dpd{\varphi}{\mathbf n} &= 0 \quad \text{ on } \ \partial \Omega, \\
\dpd{\Delta \varphi}{\mathbf n} &= 0 \quad \text{ on } \ \partial \Omega, \\
\int_{\Omega} \varphi dx &= 0.
\end{align*}
The problem above has not a unique solution  for any fixed $u$: this happens,
due to the compatibility condition,
 if and only if $u \in N$. 
 Moreover,
 since $N$ is not a manifold (unless $\alpha \neq 0$) we 
 cannot require the map $\Phi: u \mapsto \Phi(u)$ to be of class $C^1$ in $N$. We shall then extend such a map $\Phi$.

\begin{proposition}
\label{prop:auxiiary_problem_to_extend_Phi}
For every $w \in L^{6/5}(\Omega)$ there exists a unique $L(w) \in \widetilde V$ solution of
\begin{align*}
\Delta^2 \varphi - \Delta \varphi - w + \overline w &= 0 \quad \text{ in } \ \Omega, \\
\dpd{\varphi}{\mathbf n} &= 0 \quad \text{ on } \ \partial \Omega, \\
\dpd{\Delta \varphi}{\mathbf n} &= 0 \quad \text{ on } \ \partial \Omega, \\
\int_{\Omega} \varphi dx &= 0.
\end{align*} 
The map $L: L^{6/5}(\Omega) \longrightarrow \widetilde V$ is linear and continuous, hence of class $C^{\infty}$.
\end{proposition}
\begin{proof}
The weak solutions to the problem are functions $\varphi$ in the Hilbert space  $\widetilde V$ such that 
$$
\int_{\Omega} \Delta \varphi \Delta v dx + \int_{\Omega} \nabla \varphi  \nabla v dx - \int_{\Omega} wv dx = 0 \qquad \forall v \in \widetilde V.
$$
So the result follows by applying the Riesz Theorem since
the bilinear form $b: \widetilde V \times \widetilde V \longrightarrow \mathbb R$ given by
$$
b(\varphi, v) = \int_{\Omega} \Delta \varphi \Delta v dx + \int_{\Omega} \nabla \varphi  \nabla v dx.
$$
is just the  scalar product in $\widetilde V$. 
\end{proof}

The following proposition follows from well-known properties of Nemytsky operators.

\begin{proposition}
The map
$$
u \in L^6(\Omega) \mapsto qu^2 \in L^{6/5}(\Omega)
$$
is of class $C^1$.
\end{proposition}

As a consequence of the previous propositions, we can define the following map:

$$
\Phi: u \in H_0^1 (\Omega) \mapsto L(qu^2) \in \widetilde V.
$$

It is clear that 
$$
\Phi(u) = \Phi(-u) = \Phi(|u|).
$$
Moreover, for every $(u, \varphi) \in H_0^1(\Omega) \times \widetilde V$ we have that $\varphi = \Phi(u)$ if and only if for every $\eta \in \widetilde V$
\begin{equation*}
\int_{\Omega} \Delta \varphi \Delta \eta dx + \int_{\Omega} \nabla \varphi \nabla \eta dx= \int_{\Omega} qu^2 \eta dx.
\end{equation*}
Taking $\eta = \Phi(u)$ we have  in particular the important relation
\begin{equation}
\label{eq:identity_2}
\int_{\Omega} (\Delta \Phi(u))^2 dx + \int_{\Omega} |\nabla \Phi(u)|^2 dx =  \int_{\Omega} qu^2 \Phi(u) dx.
\end{equation}
The right hand side above is the {\sl interaction energy} term. 
Then we infer
\begin{align*}
||\Phi(u)||_{\widetilde V}^2 & \leq ||q||_\infty \int_{\Omega} u^2 \Phi(u) dx \\
& \leq c ||u||_4^2 ||\Phi(u)||_2 \\
& \leq c ||\nabla u||_2^2 ||\nabla \Phi(u)||_2 \\
& \leq c ||\nabla u||_2^2 ||\Phi(u)||_{\widetilde V}
\end{align*}
and hence 
\begin{equation}
\label{eq:boundedness_of_Phi}
||\Phi(u)||_{\widetilde V} \leq c ||\nabla u||_2^2,
\end{equation}
that is, $\Phi$ is bounded on bounded sets. 
We have

\begin{lemma}
\label{lem:compactness_of_Phi}
If $u_n \rightharpoonup u$ in $H_0^1(\Omega)$ then
$$
\int_{\Omega} q u_n^2 \Phi(u_n) dx \to \int_{\Omega} q u^2 \Phi(u) dx.
$$
Moreover the map $\Phi$ is compact.
\end{lemma}

\begin{proof}
Let $u_n \rightharpoonup u$ in $H_0^1(\Omega)$ and define $B_n, B: \widetilde V \longrightarrow \mathbb R$ by 
$$
B_n(\eta) := \int_{\Omega} qu_n^2\eta dx, \qquad B(\eta) := \int_{\Omega} qu^2\eta dx.
$$
Such operators are continuous due to the H\"{o}lder's inequality. For example:
\begin{align*}
\envert{\int_{\Omega} q u^2 \eta dx} &
\leq ||q||_\infty ||u||_4^2 ||\eta||_2 \leq c ||\nabla \eta||_2 \leq c ||\eta||_{\widetilde V}
\end{align*}
(where here $c$ depends on $u$).

Due to the compact embedding of $H^{1}_{0}(\Omega)$ into $L^{p}(\Omega)$ for $p\in[1,6)$,
we get $u_{n}^{2} \to u^{2}$ in $L^{6/5}(\Omega)$
%
and then
\begin{align*}
|B_n(\eta) - B(\eta)| & \leq ||q||_\infty ||u_n^2-u^2||_{6/5} ||\eta||_6 \\
& \leq c ||q||_\infty ||u_n^2-u^2||_{6/5} ||\eta||_{\widetilde V}.
\end{align*}
Hence
$$
||B_n - B|| \leq \sup_{\eta \neq 0} \frac{c ||u_n^2 - u^2||_{6/5} ||\eta||_{\widetilde V}}{||\eta||_{\widetilde V}} \to 0,
$$
namely $B_n \to B$ as operators in  $\widetilde V$. 

On the other hand, we have that $\Phi(u_n) \rightharpoonup \Phi(u)$ in $\widetilde V$. Indeed, let $g \in 
\widetilde{ V}^{'}$. Then there is some $v_g \in \widetilde V$ such that 
$$
g(\Phi(u_n)) = \int_{\Omega} \nabla \Phi(u_n)  \nabla v_g dx + \int_{\Omega} \Delta \Phi(u) \Delta v_g dx = \int_{\Omega} q u_n^2 v_g dx.
$$
But then
\begin{align*}
g(\Phi(u_n)) - g(\Phi(u)) & = \int_{\Omega} q (u_n^2 - u^2) v_g dx \\
& \leq ||q||_\infty ||u_n^2 - u^2||_2 ||v_g||_2 \to 0
\end{align*}
since $u_n^2 \to u^2$ in $L^2(\Omega)$ as well.

We then conclude that 
$$
\int_{\Omega} q u_n^2 \Phi(u_n) dx \to \int_{\Omega} q u^2 \Phi(u) dx 
$$
and by  \eqref{eq:identity_2} that $\| \Phi(u_{n})\|_{\widetilde V} \to \|\Phi(u)\|_{\widetilde V}$.
Consequently  $\Phi(u_n) \to \Phi(u)$ in $\widetilde V$. 
\end{proof}

Note that for every $u \in N$ we have that $F_\varphi'(u, \Phi(u)) = 0$. Indeed, $\Phi(u)$ is the unique solution to the problem in Proposition \ref{prop:auxiiary_problem_to_extend_Phi} with $w = qu^2$.

We now define the reduced functional of a single variable:
$$
\fullfunction{J}{H_0^1(\Omega)}{\mathbb R}{u}{F(u, \Phi(u))}
$$
With the notation $\varphi_u := \Phi(u)$ the functionl $J$ is explicitly given by (recall \eqref{eq:identity_2})
\begin{align*}
J(u) & = \frac12 \int_{\Omega} |\nabla u|^2 dx + \frac12 \int_{\Omega} q \varphi_u u^2 dx + \frac12 \int_{\Omega} q \chi u^2 dx - \frac{\kappa}{p} \int_{\Omega} |u|^p dx \\
& \quad - \frac{1}{4} \int_{\Omega} (\Delta \varphi_u)^2 dx - \frac{1}{4} \int_{\Omega} |\nabla \varphi_u|^2 dx - \frac{\alpha}{2 |\Omega|}\int_{\Omega} \varphi_u dx \\
& = \frac12 \int_{\Omega} |\nabla u|^2 dx + \frac14 \int_{\Omega} (\Delta \varphi_u)^2 dx + \frac14 \int_{\Omega} |\nabla \varphi_u|^2 dx + \int_{\Omega} q \chi u^2 dx \\
& \quad - \frac{\kappa}{p} \int_{\Omega} |u|^p dx.
\end{align*}

We note that $J$ is of class $C^1$ on $H_0^1 (\Omega)$ and even. Moreover, for every $u \in M$ we have that 
$$
J'(u)[v] = F_u'(u, \varphi_u)[v] + F_\varphi'(u, \varphi_u)[\Phi'(u)[v]] = F_u'(u, \varphi_u)[v] \quad \forall v \in H_0^1(\Omega)
$$
and hence we deduce the following

\begin{theorem}
The pair $(u, \varphi) \in M \times \widetilde V$ is a critical point of $F$ constrained on $M \times \widetilde V$ if and only if $u$ is a critical point of $J|_M$ and $\varphi = \Phi(u)$. 
\end{theorem}

\section{Proof of the main result} \label{sec:final}

The next lemma will be useful.

\begin{lemma}\label{lem:ps}
Let $D$ be a regular domain of $\mathbb R^N$ and
\begin{equation*}
\label{eq:lemma_sobolev_spaces_1}
1 \leq s \leq N,
\end{equation*}
\begin{equation*}
\label{eq:lemma_sobolev_spaces_2}
s < p < s^* = \frac{Ns}{N - s}
\end{equation*}
and 
\begin{equation*}
\label{eq:lemma_sobolev_spaces_3}
0 < r \leq N \left(1 - \frac{p}{s^*} \right).
\end{equation*}
Then there exists a constant $C > 0$ such that for every $u \in W^{1, s}(D)$ it holds that 
\begin{equation*}
\label{eq:lemma_sobolev_spaces_conclusion}
||u||_p^p \leq C ||u||_{W^{1, s}}^{p - r}||u||_s^r
\end{equation*}
\end{lemma}

\begin{proof}
See \cite[Lemma 3.1]{PisaniSiciliano2007}.
\end{proof}

\begin{remark}
If $D$ is bounded, then the conclusion of the lemma is true also in the case $p \in [1, s]$ with $r < p$. Also, if $D$ is bounded and $u \in W_0^{1, p} (D)$, then, by Poincar\'e inequality, 
$$
||u||_p^p \leq C ||\nabla u||_s^{p - r} ||u||_s^r.
$$
\end{remark}

The following lemma gives the existence of solutions to our modified problem.

\begin{lemma}\label{lem:minimo}
\label{lem:coercivity_and_weak_lower_semicontinuity}
The functional $J$ on $M$ is weakly lower semicontinuous and coercive. In particular, it has a minimum $u \in M$,
and it can be assumed positive.
\end{lemma}

\begin{proof}
We have
\begin{align*}
J(u) & = \frac12 \int_{\Omega} |\nabla u|^2dx + \frac14 \int_{\Omega} (\Delta \varphi_u)^2dx + \frac14 \int_{\Omega} |\nabla \varphi_u|^2dx 
 + \int_{\Omega} q \chi u^2 dx - \frac{\kappa}{p} \int_{\Omega} |u|^pdx\\
& \geq  \frac12 \int_{\Omega} |\nabla u|^2 dx - \|q\|_{\infty}\|\chi\|_{\infty}
- \frac{\kappa}{p} \int_{\Omega} |u|^pdx.
\end{align*}
Finally, we apply Lemma \ref{lem:ps} with $s = 2$ and $N = 3$. Since $p \in (2, 10/3)$ it holds that 
$$
p - 2 < 3 \left(1 - \frac p6 \right) < 2
$$
and we can choose
$$
p - 2 < r < 3 \left(1 - \frac p6 \right),
$$
so that by the Lemma it follows that 
$$
\frac{\kappa}{p} \int_{\Omega}|u|^p dx \leq c||\nabla u||_2^{p - r}.
$$
Hence, 
$$
J(u) \geq \frac12 \int_{\Omega} |\nabla u|^2 dx - \|q\|_{\infty}\|\chi\|_{\infty}- c'||\nabla u||_2^{p - r} 
$$
and thus $J$ is coercive and bounded from below on $M$.

Now, let $\{u_n\} \subset M$ such that $u \rightharpoonup u$. Since $M$ is weakly closed, $u \in M$. By Lemma \ref{lem:compactness_of_Phi} we know that
$$
\frac14 \int_{\Omega} (\Delta \varphi_{u_n})^2 dx + \frac14 \int_{\Omega} |\nabla \varphi_{u_n}|^2 dx \to \frac14 \int_{\Omega} (\Delta \varphi_u)^2 dx + \frac14 \int_{\Omega} |\nabla \varphi_u|^2 dx.
$$
We also know that $u_n^2 \to u^2$ in $L^{6/5}(\Omega)$ so
\begin{align*}
\int_{\Omega} q \chi (u_n^2 - u^2) dx & \leq c \int_{\Omega} u_n^2 - u^2 dx \\
& \leq c ||u_n - u||_{6/5} \to 0.
\end{align*}
Finally, the first and last terms are the norms of $u$ in $H_0^1(\Omega)$ and $L^p (\Omega)$ (up to constants), so they are weakly lower semicontinuous. 

Thus $J$ is weakly lower semicontinuous and the  existence of the minimum follows by standard results.
Note that $J(u) = J(|u|)$ so the minimum may be assumed to be positive. 
\end{proof}


We will use a deformation argument to show that there are infinitely many solutions. A crucial point is that the functional satisfies the Palais-Smale condition.
We recall that in general, it is said that the $C^{1}$ functional $\mathcal I$
satisfies the Palais-Smale condition on the manifold $\mathcal M$, if 
any sequence $\{u_{n}\}\subset \mathcal M$ such that
$\left\{\mathcal I(u_{n})\right\}$ is bounded and   $\mathcal I(u_{n}) \to0$
in the tangent bundle, admits a convergent subsequence to an element $u\in \mathcal M$.

\begin{proposition}
\label{prop:Palais_Smale}
The functional $J$ satisfies the Palais-Smale condition on $M$.
\end{proposition}

\begin{proof}
Let $\{u_n\} \subset M$ be 
such that
$$
\{J(u_n)\} \ \text{ is bounded }
$$
and 
\begin{equation}
\label{eq:derivative_goes_to_zero}
J|_M'(u_n) \to 0.
\end{equation}
By \eqref{eq:derivative_goes_to_zero}  there exists two sequences of real numbers $\{\lambda_n)\}, \{\beta_n \} $ 
and a sequence $\{v_{n}\} \subset H^{-1}$ such that $v_n \to 0$ and 
\begin{equation}
\label{eq:derivative_goes_to_zero_meaning}
- \Delta u_n + q(\varphi_n + \chi)u_n - \kappa |u_n|^{p - 2}u_n = \lambda_n u_n + \beta_n q u_n + v_n
\end{equation}
where $\varphi_n := \varphi_{u_n}$.

Since $J$ is coercive and $\{J(u_n)\}$ is bounded, we know that $\{u_n\}$ is bounded in $H_0^1 (\Omega)$. Hence there exists $u \in H_0^1(\Omega)$ such that $u_n \rightharpoonup u$, up to a subsequence. By the compact embeddings and Lemma \ref{lem:compactness_of_Phi} we know that 
\begin{equation}\label{eq:conv}
u_{n}\to u \quad\text{in } \  L^{p}(\Omega), \quad 
\varphi_n \to \varphi_{u}\quad  \text{ in } \  H^{2}(\Omega).
\end{equation}
 Also, since $M$ is weakly closed, we know that $u \in M$. It only remains to show that $u_n \to u$ in $H_0^1 (\Omega)$.

By \eqref{eq:derivative_goes_to_zero_meaning} we have that 
\begin{equation}
\label{eq:derivative_goes_to_zero_applied_to_un}
\frac12 \int_{\Omega} |\nabla u_n|^2 dx + \frac12 \int_{\Omega} q(\varphi_n + \chi)u_n^2 dx - \frac{\kappa}{p} \int_{\Omega} |u_n|^p dx - \langle v_n, u_n \rangle = \lambda_n + \alpha \beta_n.
\end{equation}
By \eqref{eq:conv} we infer
\begin{eqnarray*}\label{eq:}
\left| \int_{\Omega} \left(q(\varphi_n + \chi)u_n^2 -  q(\varphi_u + \chi)u^2  \right)dx\right| &\leq&
c \int_{\Omega} \left| \varphi_{n}+\chi \right| |u_{n}^{2} - u^{2}| dx +
\int_{\Omega} u^{2} \left| \varphi_{n} - \varphi \right| dx \\
&=&o_{n}(1)
\end{eqnarray*}
where we are denoting with $o_{n}(1)$ a vanishing sequence.
%
Then the right-hand side of \eqref{eq:derivative_goes_to_zero_applied_to_un} is bounded
and we can assume that
$$
\lambda_n + \alpha \beta_n =\xi + o_{n}(1)
$$
with $\xi \in \mathbb R$.
Then \eqref{eq:derivative_goes_to_zero_meaning} becomes 
\begin{equation}
\label{eq:derivative_goes_to_zero_becomes}
-\Delta u_n + q(\varphi_n + \chi)u_n - \kappa |u_n|^{p - 2}u_n - v_n =(\xi + o(1))u_n - \beta_n(q - \alpha)u_n.
\end{equation}

Now, since $u \in M$ we know that $||u||_2 = 1$. This, together with the assumption $|q^{-1}(\alpha)| = 0$ implies that $(q - \alpha)u$ is not identically zero. Then there exists a test function $w \in C^{\infty}_{0}(\Omega)$ such that 
$$
\int_{\Omega} (q - \alpha)u w dx \neq 0.
$$
Evaluating \eqref{eq:derivative_goes_to_zero_becomes} on this $w$ we get
\begin{multline}\label{eq:beta}
\int_{\Omega} \nabla u_n  \nabla w dx + \int_{\Omega} q(\varphi_n + \chi)u_n w dx -\kappa \int_{\Omega}|u_n|^{p - 2}u_n w dx \\ - \langle v_n, w \rangle - (\lambda + o_{n}(1))\int_{\Omega} u_n w dx
= \beta_n \int_{\Omega} (q - \alpha)u_n w dx
\end{multline}
and using again \eqref{eq:conv} we see that
every term in the left-hand side converges. 
Also, by the weak convergence of $\{u_n\}$, 
$$
\int_{\Omega} (q - \alpha)u_n w dx \to \int_{\Omega} (q - \alpha)u w dx.
$$
This implies, coming back to \eqref{eq:beta}, that $\{\beta_n\}$ is bounded, which in turn implies that $\{\lambda_n\}$ is bounded. 

Applying \eqref{eq:derivative_goes_to_zero_becomes} to $u_n - u$ we get
\begin{multline}\label{eq:ultimaconv}
\int_{\Omega} \nabla u_n \nabla(u_n - u) dx + \int_{\Omega} q(\varphi_n + \chi)u_n(u_n - u) dx - \kappa \int_{\Omega} |u_n|^{p - 2}u_n (u_n - u)\\
- \langle v_n, u_n - u \rangle = (\lambda + o(1))\int_{\Omega} u_n (u_n - u) dx + \beta_n \int_{\Omega} (q - \alpha)u_n(u_n - u) dx. 
\end{multline}
Since (again by \eqref{eq:conv}) we have
\begin{align*}
& \int_{\Omega} q(\varphi_n + \chi)u_n(u_n - u) dx \to 0, \quad \langle v_n, u_n - u \rangle \to 0, \\
& \int_{\Omega}|u_n|^{p - 2} u_n(u_n - u) dx \to 0, \quad (\lambda + o(1))\int_{\Omega} u_n (u_n - u) dx \to 0
\end{align*}
and
$$
\beta_n \int_{\Omega} (q - \alpha)u_n(u_n - u) dx \to 0,
$$
we conclude by \eqref{eq:ultimaconv}
 that $||\nabla u_n||_2 \to ||\nabla u||_2$ and so $u_n \to u$ in $H_0^1(\Omega)$.
%
%
%
%
%
%
%
%
\end{proof}

\medskip

Now we can give the proof of Theorem \ref{thm:main_theorem}.

By Theorem \ref{thm:existence_of_sets_of_genus_k}, $M$ has compact, symmetric subsets of genus $k$ for every $k \in \mathbb N$. 

Let us recall now a classical result in critical point theory. We give the proof for the reader convenience.
\begin{lemma}
\label{lem:sublevels_have_finite_genus}
For any $b \in \mathbb R$ the sublevel
$$
J^b = \cbr{u \in M \ : \ J(u) \leq b}
$$
has finite genus.
\end{lemma}

\begin{proof}
We argue by contradiction. Suppose that
$$
D = \cbr{b \in \mathbb R \ : \ \gamma(J^b) = \infty} \neq \emptyset.
$$
Since $J|_M$ is bounded from below, then $D$ is bounded from below. Then 
$$
- \infty < \overline b = \inf D < \infty.
$$
Moreover, since $J|_M$ satisfies the Palais-Smale condition, the set 
$$
Z = \cbr{u \in M \ : \ J(u) = \overline b, \ J|_M'(u) = 0}
$$
is compact. Hence there exists a closed symmetric neighborhood $U_Z$ of $Z$ such that $\gamma(U_Z) < \infty$. By the Deformation Lemma, there exists an $\varepsilon > 0$ such that $J^{\overline b -  \varepsilon}$ includes a deformation retract of $J^{\overline b + \varepsilon} \setminus U_Z$. Then, by the properties of the genus, 
$$
\gamma(J^{\overline b + \varepsilon}) \leq \gamma(J^{\overline b + \varepsilon} \setminus U_Z) + \gamma(U_Z) \leq \gamma(J^{\overline b - \varepsilon}) + \gamma(U_Z) < \infty, 
$$
a contradiction.
\end{proof}

Let $n \in \mathbb N$. By Lemma \ref{lem:sublevels_have_finite_genus} there exists some $k \in \mathbb N$ depending on $n$ such that 
$$
\gamma(J^n) = k.
$$
Let
$$
A_{k+1} = \cbr{A \subset M \ : \ A = -A, \overline A = A, \gamma(A) = k + 1}
$$
that we know is not empty by   Theorem \ref{thm:existence_of_sets_of_genus_k}.

 By the monotonicity property of the genus, any $A \in A_{k+1}$ is not contained in $J^n$, then $\sup_A J > n$ and therefore 
$$
c_n = \inf_{A \in A_{k+1}} \sup_{u \in A} J(u) \geq n.
$$
Well known results (see e.g. \cite{Szulkin1988})
say that $c_{n}$ are critical levels for $J|_{M}$ and then there is a
sequence $\{u_{n}\}$ of critical points such that
$$J(u_{n})= c_{n}\to +\infty$$

The critical points give rise to Lagrange multipliers $\omega_n, \mu_n$ and then, recalling the decomposition $\varphi = \phi - \chi - \mu$, to solutions $(u_n, \omega_n, \phi_n) \in H_0^1(\Omega) \times \mathbb R \times H^2 (\Omega)$ of the original problem.

\medskip

We show that $||\nabla u_n||_2 \to + \infty$.
Since
$$
\int_{\Omega} q \chi u_n^2 dx \leq ||q \chi||_\infty,
$$
and by \eqref{eq:boundedness_of_Phi} it is
$$
||\varphi_n||_{\widetilde V} = \int_{\Omega} (\Delta \varphi_n)^2 dx + \int_{\Omega} |\nabla \varphi_n|^2 dx \leq c ||\nabla u_{n}||_2^2, 
$$
we see that
$$
|J(u_n)| \leq (1+c) ||\nabla u_n||^2 + c' ||\nabla u_n||_2^{p } + ||q \chi||_\infty
$$
and then $\{u_{n}\}$ can not be  bounded.

This concludes the proof of Theorem \ref{thm:main_theorem}.

\end{document}